\documentclass[11pt]{amsart}

\usepackage{amsrefs}

\usepackage{epsfig}  		
\usepackage{epic,eepic}       
\usepackage{graphicx}

\usepackage{enumerate}
\usepackage{mathbbol}
\usepackage{amssymb}
\usepackage{braket}
\usepackage{longtable}
\usepackage[colorlinks]{hyperref}
\usepackage{caption}
\usepackage{doi}
\usepackage{mathtools}
\usepackage{microtype}

\DeclareSymbolFontAlphabet{\mathbb}{AMSb}
\DeclareSymbolFontAlphabet{\mathbbl}{bbold}

\newenvironment{claimproofend}[1][Proof of Claim]{\begin{proof}[#1]}{\end{proof}}

\newtheorem{lemma}{Lemma}[section]
\newtheorem*{lemma*}{Lemma}
\newtheorem{theorem}[lemma]{Theorem}
\newtheorem*{theorem*}{Theorem}
\newtheorem{corollary}[lemma]{Corollary}
\newtheorem{proposition}[lemma]{Proposition}

\newtheorem*{proposition*}{Proposition}
\newtheorem{fact}[lemma]{Fact}
\newtheorem*{fact*}{Fact}

\newtheorem*{notation*}{Notation}
\newtheorem*{conventions*}{Conventions}

\newtheorem*{remark*}{Remark}

\newtheorem*{corollary*}{Corollary}

\newtheorem*{conjecture*}{Conjecture}

\newtheorem*{problem*}{Problem}

\newtheorem*{question*}{Question}

\newtheorem{assumption*}{Assumption}

\theoremstyle{definition}

\newtheorem*{example*}{Example}
\newtheorem{definition}[lemma]{Definition}
\newtheorem*{definition*}{Definition}

\theoremstyle{remark}

\newtheorem*{claim*}{Claim}

\newtheorem*{case*}{Case}
\newtheorem*{construction*}{Construction}

\newtheorem*{exercise*}{Exercise}

\numberwithin{equation}{section}

\newcommand{\Z}{\mathbb{Z}}

\newcommand{\Q}{\mathbb{Q}}

\newcommand{\bs}{\backslash}

\newcommand\FF{{\mathcal F}}

\newcommand\RR{{\mathcal R}}

\newcommand\acl{\hbox{\rm acl}}

\newcommand\<{\langle}
\renewcommand\>{\rangle}

\newcommand{\tp}{\mathrm{tp}}
\newcommand{\qftp}{\mathrm{qftp}}

\newcommand{\Ra}{\Rightarrow}

\def\Ind#1#2{#1\setbox0=\hbox{$#1x$}\kern\wd0\hbox to 0pt{\hss$#1\mid$\hss}
	\lower.9\ht0\hbox to 0pt{\hss$#1\smile$\hss}\kern\wd0}

\def\ind{\mathop{\mathpalette\Ind{}}}

\def\notind#1#2{#1\setbox0=\hbox{$#1x$}\kern\wd0
	\hbox to 0pt{\mathchardef\nn=12854\hss$#1\nn$\kern1.4\wd0\hss}
	\hbox to 0pt{\hss$#1\mid$\hss}\lower.9\ht0 \hbox to 0pt{\hss$#1\smile$\hss}\kern\wd0}

\def\nind{\mathop{\mathpalette\notind{}}}

\def\includeE#1{{\lhook\kern-3.5pt\joinrel\smash{
			\mathop{\longrightarrow}\limits^{#1}}}}

\def\efor/{Example~\ref{E4}}

\def\BL/{Baldwin--Lachlan}
\def\Bu/{Buechler}
\def\Hr/{Hrushovski}
\def\lm/{locally modular}
\def\wm/{weakly minimal}
\def\nm/{non--modular}
\def\ss/{superstable}
\def\ud/{unidimensional}
\def\sm/{strongly minimal}

\def\abar{\bar{a}}

\def\dbar{\bar{d}}
\def\ebar{\bar{e}}

\def\hbar{\bar{h}}

\def\rbar{\bar{r}}
\def\sbar{\bar{s}}

\def\wbar{\bar{w}}
\def\xbar{\bar{x}}
\def\ybar{\bar{y}}
\def\zbar{\bar{z}}
\def\Mbar{\bar{M}}

\def\acl{{\rm acl}}
\def\acl{{\rm acl}}

\def\stp{{\rm stp}}

\def\tr/{trivial}
\def\nt/{non--trivial}
\def\st/{strong type}

\def\abar{\bar{a}}

\def\dbar{\bar{d}}
\def\ebar{\bar{e}}
\def\A{{\mathcal A}}

\def \F{{\mathcal F}}

\def\Q{{\mathbb Q}}

\def\Z{{\mathbb Z}}
\def\stp{{\rm stp}}

\def\acl{{\mathrm acl}}

\def\Fa0{{\FF^a_{\aleph_0}}}

\def\<{\langle}
\def\>{\rangle}

\newbox\smilebox
\newbox\anchorbox
\newbox\noanchorbox
\newbox\tempbox

\setbox\smilebox=\hbox{$\smile$}

\def\anchor{\hbox{\vtop{
			\hbox to \wd\smilebox{\hfil\vrule width.4pt height7pt depth1pt\hfil}
			\vskip  -11.5truept
			\hbox to \wd\smilebox{\hfil$\smile$\hfil}}}}
\setbox\anchorbox=\anchor
\def\noanchor{\hbox{\vtop{
			\hbox to \wd\anchorbox{\hfil\anchor\hfil}
			\vskip -14truept
			\hbox to \wd\anchorbox{\hfil/\hfil}}}}
\setbox\noanchorbox=\noanchor

\def\fg#1#2#3{\setbox\tempbox=\hbox{$\scriptstyle{#2}$}
	\ifnum\wd\anchorbox>\wd\tempbox\dimen255=\wd\anchorbox
	\else\dimen255=\wd\tempbox\fi
	{#1\,\vtop{\hbox to \dimen255{\hfil\anchor\hfil}
			\vskip -6truept
			\hbox to \dimen255{\hfil$\scriptstyle{#2}$\hfil}}
		\,#3}}

\def\nfg#1#2#3{\setbox\tempbox=\hbox{$\scriptstyle{#2}$}
	\ifnum\wd\noanchorbox>\wd\tempbox\dimen255=\wd\noanchorbox
	\else\dimen255=\wd\tempbox\fi
	{#1\,\vtop{\hbox to \dimen255{\hfil\noanchor\hfil}
			\vskip -6truept
			\hbox to \dimen255{\hfil$\scriptstyle{#2}$\hfil}}
		\,#3}}

\setbox1=\hbox{$\bot$}

\def\north#1#2{#1\,
	\hbox{$\bot$\llap {\hbox to\wd1 {\hfil $/$\hfil}}}
	\,#2}

\def\nao#1#2#3{#1\  \hbox{\vtop{ 
			\baselineskip=4pt
			\hbox{$\bot$\llap {\hbox to\wd1 {\hfil $/$\hfil}}
				\hskip .05em \llap{\hbox{$^{\scriptscriptstyle{a}}$}}}\hbox{$\scriptstyle
				{#2}$}}}\, #3}

\def\includeE#1{{\lhook\kern-3.5pt\joinrel\smash{
			\mathop{\longrightarrow}\limits^{#1}}}}

\def\efor/{Example~\ref{E4}}

\def\BL/{Baldwin--Lachlan}
\def\Bu/{Buechler}
\def\Hr/{Hrushovski}
\def\lm/{locally modular}
\def\wm/{weakly minimal}
\def\nm/{non--modular}
\def\ss/{superstable}
\def\ud/{unidimensional}
\def\sm/{strongly minimal}

\def\abar{\overline{a}}

\def\dbar{\overline{d}}
\def\ebar{\overline{e}}

\def\hbar{\overline{h}}

\def\rbar{\overline{r}}
\def\sbar{\overline{s}}

\def\wbar{\overline{w}}
\def\xbar{\overline{x}}
\def\ybar{\overline{y}}
\def\zbar{\overline{z}}
\def\Mbar{\overline{M}}

\def\acl{{\rm acl}}

\def\tp{{\rm tp}}
\def\stp{{\rm stp}}

\def\tr/{trivial}
\def\nt/{non--trivial}
\def\st/{strong type}

\def\TV/{Tarski--Vaught}

\def\sc/{sound construction}
\def\ac/{atomic construction}

\def\fal/{functional}
\def\upl/{unique parallel lines}
\def\chp/{categorical in a higher power}

\title{Worst case expansions of complete theories}
\author{Samuel Braunfeld}
\author{Michael C. Laskowski$^*$}
\thanks{$^*$Partially supported
by NSF grant DMS-1855789}

    \def\E{{\mathcal E}}

    \def\CC{{\mathfrak C}}
    \def\G{{\mathcal G}}

\begin{document}

\begin{abstract}
	Given a complete theory $T$ and a subset $Y \subseteq X^k$, we precisely determine the {\em worst case complexity}, with respect to further monadic expansions, of an expansion $(M,Y)$ by $Y$ of a model $M$ of $T$ with universe $X$.  
	In particular, although by definition monadically stable/NIP theories are robust under arbitrary monadic expansions, 
	we show that monadically NFCP (equivalently, mutually algebraic) theories are the largest class that is robust under anything beyond monadic expansions. We also exhibit a paradigmatic  structure for the failure of each of monadic NFCP/stable/NIP  and prove each of these paradigms definably embeds into a monadic expansion of a sufficiently saturated model of any theory without the corresponding property.  
\end{abstract}

\maketitle

\section{Introduction}

The idea of measuring the complexity of a first order theory by determining the worst case complexity of its models under expansions by arbitrarily many unary (monadic) predicates was introduced by Baldwin and Shelah in \cite{BS}. For example, the theory ACF of algebraically closed fields is maximally complex with respect to this measure, even though it is classically very simple and has many well-studied tame monadic expansions. One way to see this complexity is to first name an infinite linearly independent set by a unary predicate $A$; then any graph $G$ with vertex set $A$ is definable in the further expansion by the unary predicate $B_G = \set{g + h : g, h \in A, (g,h) \text{ is an edge in $G$}}$. As any structure in a finite language is definable in a monadic expansion of a graph (e.g., by the construction in \cite[Theorem 5.5.1]{Hodges}), we may for example define models of ZFC in monadic expansions of models of ACF.

In contrast to ACF, some theories such as $Th(\Q, <)$ are monadically NIP, i.e. no monadic expansion has the independence property. (The definitions of NIP, as well as stability and NFCP, are recalled in the next section.)  If a theory is not monadically NIP then it can define arbitrary graphs in unary expansions of its models, as ACF does, and thus is also maximally complex by our measure. Similarly, there exist monadically stable theories such as the theory of an equivalence relation with infinitely many infinite classes, and monadically NFCP theories (which coincide with the mutually algebraic theories of \cite{JSL}) such as $Th(\Z, succ)$.

Our first result shows that the random graph, $(\Q, <)$, and the equivalence relation with infinitely many infinite classes are paradigms of structures that respectively are not monadically NIP/stable/NFCP, in the sense that we may define these paradigms on singletons in a monadic expansion of any sufficiently saturated model without the corresponding property (Theorem \ref{paradigm}).

For our main result, recall  that while monadically NIP and monadically stable theories are closed under monadic expansions by definition, the monadically NFCP theories satisfy a stronger closure property: if $T$ is monadically NFCP and $M \models T$, then any expansion of $M$ by arbitrarily many relations definable in monadically NFCP structures with the same universe as $M$ remains monadically NFCP \cite{JSL}.
Our main result proves  that any attempt to extend these closure statements to larger classes of relations fails spectacularly, producing expansions of models defining arbitrary graphs.

Before stating our main theorem, we must introduce an extremely simple class of theories. 

\begin{definition} \label{pur mon}  A complete theory $T$ is {\em purely monadic} if, for every model $M\models T$ with universe $\lambda$, every definable (with parameters) 
	$Y\subseteq\lambda^k$ is definable in a monadic structure $(\lambda,U_1,\dots,U_n)$.
\end{definition}

\begin{theorem}  \label{big} Suppose a complete theory $T$ is not purely monadic and $Y\subseteq\lambda^k$ is not definable in a purely monadic structure, where $|\lambda| \geq |T|$.
	
	If either $T$ is not monadically NFCP or $Y$ is not definable in a monadically NFCP structure, then there is $M \models T$ with universe $\lambda$ such that the expansion $(M, Y)$ is not monadically NIP.
	
	Otherwise, if $T$ is monadically NFCP and $Y$ is definable in a monadically NFCP structure, then for every $M \models T$ with universe $\lambda$, the expansion $(M, Y)$ is monadically NFCP.
\end{theorem}

The cases ruled out by the hypotheses of this theorem are straightforward, and are handled by Fact \ref{restbig}. 

 Section 3 is dedicated to the result on paradigmatic failures of monadic properties mentioned above, while in Section 4 we find a canonical configuration present in any structure that is monadically NFCP but not purely monadic. In Section 5, Theorem \ref{big} is then proved in cases, by suitably overlaying the available configurations to monadically define arbitrary graphs.

\subsection{Acknowledgments} We thank the referee for a careful reading, yielding corrections and suggestions for clarifying the exposition.

\section{Preliminaries}

We recall the following standard conditions on a partitioned formula $\phi(\xbar,\ybar)$, when we are working in a sufficiently saturated model $\CC$ of a complete theory $T$:
$\phi(\xbar,\ybar)$ has the {\em finite cover property (FCP)} if, for arbitrarily large $n$, there are $\<\abar_i:i<n\>$ in $\CC$ such that, 
$$\CC \models \neg\exists \xbar ( \bigwedge_{i<n} \phi(\xbar,\abar_i)) \wedge \bigwedge _{\ell <n} \exists \xbar ( \bigwedge_{i < n, i\neq \ell} \phi(\xbar,\abar_i))$$
$\phi(\xbar,\ybar)$ has the {\em order property} if, for each $n$, there are $\<\abar_i:i<n\>$ in $\CC$ such that, for each $k<n$,
$$\CC \models \bigwedge_{k<n}\left[\exists \xbar (\bigwedge_{i<k} \phi(\xbar,\abar_i)\wedge\bigwedge_{k\le i<n} \neg\phi(\xbar,\abar_i))\right]$$
$\phi(\xbar,\ybar)$ has the {\em independence property} if, for each $n$, there are $\<\abar_i:i<n\>$ in $\CC$ such that, 
$$\CC \models \bigwedge_{s\subseteq [n]}\left[\exists \xbar (\bigwedge_{i \in s} \phi(\xbar,\abar_i)\wedge\bigwedge_{i \in n \bs s} \neg\phi(\xbar,\abar_i))\right]$$

A complete theory $T$ is {\em NFCP} if no partitioned formula $\phi(\xbar,\ybar)$ has the FCP, $T$ is {\em stable} if no partitioned formula $\phi(\xbar,\ybar)$ has the order property, and
$T$ is {\em NIP} if no partitioned formula $\phi(\xbar,\ybar)$ has the independence property.  

It is well known that for complete theories,  NFCP$ \ \Rightarrow \ \hbox{stable} \ \Rightarrow\ $ NIP, and as purely monadic theories are NFCP (e.g., by the comment after Fact \ref{MAcrit}),  we have the following implications for a complete theory $T$.
$$\hbox{purely monadic}\quad \Rightarrow \quad \hbox{mon.\ NFCP}\quad\Rightarrow \quad\hbox{mon.\ stable}\quad\Rightarrow\quad\hbox{mon.\ NIP}$$

We now introduce some definitions for convenience.

\begin{definition}
Given a complete theory $T$, a cardinal $\lambda$, a subset $Y\subseteq\lambda^k$ for some $k\ge 1$, and a property $P$ of theories (we will be particularly interested in monadic NIP), we say {\em $(T,Y)$ is always $P$} if $Th(M,Y)$ has $P$ for all models $M$ of $T$ with universe $\lambda$.
\end{definition}

\begin{definition} \label{p definable}  A subset $Y\subseteq \lambda^k$ is {\em monadically definable} if it is definable in some monadic structure $(N,U_1,\dots,U_n)$.
	
	$Y\subseteq\lambda^k$ is {\em monadically NFCP definable} if it is definable in some monadically NFCP structure $N$.
	Analogously, $Y$ is {\em monadically stable/monadically NIP definable} if it is definable in some monadically stable/monadically NIP structure $N$.
	
	Equivalently, a subset $Y\subseteq\lambda^k$ is monadically definable (respectively, monadically NFCP/stable/NIP definable) if and only if
	the structure $N=(\lambda,Y)$ in a language with a single $k$-ary predicate symbol, is purely monadic (respectively, monadically NFCP/stable/NIP).
\end{definition}

Thus, we have the following implications for $Y \subseteq \lambda^k$.
$$\hbox{mon.\ definable}\ \Rightarrow \ \hbox{mon.\ NFCP def}\ \Rightarrow \ \hbox{mon.\ stable def}\ \Rightarrow\ \hbox{mon.\ NIP def}$$

The hypotheses of Theorem \ref{big} ruled out the cases where $T$ is purely monadic or $Y$ is monadically definable. The following fact is immediate from unpacking definitions, but we include it for completeness.

\begin{fact} \label{restbig}  Let $T$ be a complete theory, $Y\subseteq \lambda^k$, and $P \in$ \{purely monadic, monadically NFCP, monadically stable, monadically NIP\}.
	\begin{enumerate}
		\item  If $T$ is purely monadic and $Y$ is $P$ definable then $(T,Y)$ is always $P$.
		\item  If $T$ is $P$ and $Y\subseteq \lambda^k$ is monadically definable then
		$(T,Y)$ is always $P$.
	\end{enumerate}
\end{fact}

There are many equivalents to monadic NFCP (e.g., see \cites{MCLarch,JSL,BLcell}), monadic stability (see \cites{BS, TD}),  and monadic NIP (see \cites{BS,ShHanf,BLmonNIP}).  What we use is encapsulated in the rest of this section.

\begin{definition}
	Let $T$ be a complete theory.
	
	$T$ is {\em weakly minimal} if for any pair $M\preceq N$ of models, every non-algebraic 1-type $p\in S_1(M)$ has a unique non-algebraic
	extension $q\in S_1(N)$.
	
	$T$ is {\em (forking) trivial} if whenever $\set{A,B,C}$ is pairwise forking-independent over $D$, then it is an independent set over $D$.
	
	$T$ is {\em totally trivial} if for all $A,B,C,D$, if $A \ind_D B$ and $A \ind_D C$ then $A \ind_D BC$. (This is obtained from the definition of triviality by removing the hypothesis that $B \ind_D C$.)
\end{definition}

\begin{fact}[\cite{JSL}*{Theorem~3.3}]   \label{MAchar} The following are equivalent for a complete theory $T$.
	\begin{enumerate}
		\item  $T$ is monadically NFCP.
		\item  $T$ is mutually algebraic (see Definition~\ref{MA} below).
		\item  $T$ is weakly minimal and trivial.
	\end{enumerate}
\end{fact}

Although we will not explicitly use it, ``trivial'' could be replaced by ``totally trivial'' in $(3)$, since they are equivalent assuming weak minimality, e.g. by \cite[Proposition 5]{Triv}.

We will make use of the following sufficient condition from \cite{BS} for monadically defining arbitrary graphs, or equivalently by Fact \ref{fact:8.1.10}, for the failure of monadic NIP.

\begin{definition}  \label{coding}
  A structure $M$ {\em admits coding} if there are infinite subsets $A,B,C\subseteq M^1$ and a formula $\phi(x,y,z)$ whose restriction to $A\times B\times C$
	is the graph of a bijection $f:A\times B\rightarrow C$.  A theory $T$ {\em (monadically) admits coding} if (some monadic expansion $M^*$ of) some model $M$ of $T$
	admits coding.
	\end{definition}  


\begin{fact} [\cites{BS, BLmonNIP}] \label{fact:8.1.10}
	The following are equivalent for a complete theory $T$.
	\begin{enumerate}
		\item  $T$ is monadically NIP.
		\item  $T$ does not monadically admit coding.
		\item  There is a graph that is not definable in any monadic expansion of any model of $T$.
	\end{enumerate}
\end{fact}

\begin{fact} [\cites{BS, TD}] \label{fact:4.2.6}
	The following are equivalent for a stable complete theory $T$.
	\begin{enumerate}
		\item $T$ is monadically stable.
		\item $T$ is monadically NIP.
		\item $T$ does not admit coding.
		\item $T$ is totally trivial and forking is transitive on singletons, i.e. for all $D$, if $a \nind_D b$ and $b \nind_D c$ then $a \nind_D c$.
	\end{enumerate}
\end{fact}
\begin{proof}
	$(1) \Ra (2)$ is clear, $(2) \Ra (3)$ follows from Fact \ref{fact:8.1.10}, and $(3) \Ra (4)$ is \cite[Lemma 4.2.6]{BS}. Finally, $(4) \Ra (1)$ is essentially contained in Theorems 3.2.4 and 4.2.17 of \cite{BS}, but verifying this involves tracing through several other results. The implication is more cleanly stated in Theorems 2.17 and 2.21 of \cite{TD}, noting that what \cite[Definition 2.5]{TD} calls {\em forking-triviality} is equivalent to the two conditions in $(4)$ by some basic forking-calculus manipulations.
\end{proof}

\begin{lemma} \label{nonwm}
	If $T$ is monadically stable (equivalently, stable and monadically NIP) but not monadically NFCP, then $T$ is not weakly minimal.
\end{lemma}
\begin{proof}
	Fact \ref{fact:4.2.6} shows the parenthetical equivalence, and also shows that if $T$ is monadically stable then it is (totally) trivial. So by Fact \ref{MAchar}, if $T$ is not monadically NFCP then it cannot be weakly minimal.
\end{proof}

\section{Finding paradigms of non-monadically NFCP theories}

In this section, we show the following classical structures will always witness the failure of monadic NIP/stability/NFCP in a suitable monadic expansion.

\begin{itemize}
\item  The random graph, sometimes called the Rado graph, $\RR=(A,E)$ is the standard example of a structure whose theory has the independence property.
In particular, its theory is not monadically NIP.  

\item  Dense linear order (DLO), the theory of $(\Q,\le)$, is one of the simplest non-stable theories as $\le$ visibly witnesses the order property.  Thus, DLO is not monadically stable,
but it is monadically NIP (e.g., see \cite{Guide}*{Proposition A.2}).

\item  Let $\E=(X,E)$, where $X=\omega\times\omega$ (so each element of $X$ can be uniquely written as $(a,b)\in\omega^2$) and $E((a_1,b_1),(a_2,b_2))$ holds if and only if
$a_1=a_2$.  Thus, $\E$ is the (unique) model of the $\omega$-categorical theory of an equivalence relation with infinitely many classes, with each class infinite.  The theory $Th(\E)$ is monadically stable, but it is not monadically NFCP.  To see the former, one can check it satisfies the conditions in Fact \ref{fact:4.2.6} $(4)$. To see the latter, one can add a single unary predicate whose interpretation contains exactly $n$ elements from the $n^{{\rm th}}$ $E$-class.   This expanded structure is a paradigm of a stable structure with the finite cover property.  
\end{itemize}

We next show that these paradigms all {\em definably embed} into a monadic expansion of any model of its class.  
It is crucial to consider structures defined in $M^1$ rather than in a cartesian power, as this will allow us to name substructures in unary expansions.

\begin{definition} \label{def:de}

We say a structure $\A$ {\em definably embeds} into another structure $M$ (possibly in a different language) if $\A$ is definable on singletons in $M$. 

Explicitly, let $\A=(A,R)$ be any structure in a language with a binary relation, and let $M$ be an $L$-structure in some arbitrary language.
We say {\em $\A$ definably embeds} into $M$ if there are $L$-definable $X\subseteq M^1$ and $R'\subseteq X^2$ and a bijection $f:A\rightarrow X$ such that 
for all $a,b\in A$, $\A\models R(a,b)$ iff $M\models R'(f(a),f(b))$.  [Informally, $(X,R')$ is an `isomorphic copy of $\A$'.]
	
	A definable embedding $f:(A,R)\rightarrow (X,R')$ is {\em type-respecting} if, in addition, for any tuples $\abar,\abar'\in A^n$, if $\qftp_{\A}(\abar)=\qftp_{\A}(\abar')$, then
	$\tp_M(f(\abar))=\tp_M(f(\abar'))$.
\end{definition}

\begin{theorem} \label{paradigm}  Let $T$ be a complete $L$-theory.
\begin{enumerate}
\item  If $T$ is not monadically NIP, then the random graph $\RR$ definably embeds into some monadic expansion $M^*$ of a model $M$ of $T$.
\item  If $T$ is not monadically stable, then there is a definable, type-respecting embedding of $(\Q,\le)$
 into some monadic expansion $M^*$ of a model $M$ of $T$.
\item  If $T$ is monadically stable but not monadically NFCP, then there is a definable, type-respecting embedding of $\E$ into some monadic expansion $M^*$ of a model $M$ of $T$.
\item  If $T$ is not monadically NFCP, $\E$ definably embeds into some monadic expansion $M^*$ of a model $M$ of $T$.
\end{enumerate}
\end{theorem}

\begin{proof}  
 (1) Assume $T$ is not monadically NIP.  By either \cite{BS} or \cite{BLmonNIP}, there is a monadic expansion $M^*$ of a model of $T$ that admits coding, i.e., there are
 infinite sets $A,B,C$ and a 3-ary $L^*$-formula $\phi(x,y,z)$ coding the graph of a bijection from $A\times B$ to $C$.   By adding more unary predicates, we may assume each of 
 $A,B,C$ are definable in $M^*$ and are countably infinite, and by replacing $\phi$ by $\phi(x,y,z)\wedge A(x)\wedge B(y)\wedge C(z)$, the graph of $\phi$ is precisely the bijection.
 Now add a unary predicate $D\subseteq C$ so that for every $a_1\neq a_2 \in A$, there is a unique $b \in B$ such that $M^* \models \exists (d_1, d_2 \in D) (\phi(a_1,b,d_1) \wedge \phi(a_2,b,d_2))$.
 Thus, in this expansion, one can think of $B$ as coding (symmetric) edges of $A$ via this formula.  For the whole of $D$, we get a complete graph on $A$, but for any predetermined
 graph $\G$ with universe $A$, one can add a single unary predicate $E\subseteq D$ so that for any $a_1\neq a_2\in A$, the following formula holds iff $a_1,a_2$ are edge-related in $\G$.
\[\exists y\exists z_1 \exists z_2 (E(z_1) \wedge E(z_2) \wedge \phi(a_1,y,z_1)\wedge\phi(a_2,y,z_2) \]
In particular, we get a definable embedding of $\RR$ into this expansion of $M^*$.

(2)  By passing to a monadic expansion, we may assume $T$ itself is unstable.  [In fact, any monadically NIP, non-monadically stable theory must itself be unstable,
but we don't need this.]  By \cite{Pierre}, after adding parameters, there is a formula
$\phi(x,y)$ with the order property, where $x$ and $y$ are both singletons. Thus, by adding an additional unary predicate for each of the parameters $c$ (with interpretation $\{c\}$)
there is a monadic expansion $M^*$ of a model of $T$ with a 0-definable $L^*$-formula $\psi(x,y)$ with the order property.

By Ramsey and compactness and by passing to an $L^*$-elementary extension, we may assume there are order-indiscernible subsets
$A=\{a_i:i\in\Q\}$ and $B=\{b_j:j\in\Q\}$ of $M^*$ such that
$M^*\models \psi(a_i,b_j)$ iff $i\le j$. 
By replacing $M^*$ by a monadic expansion of itself, we may additionally assume there are predicates for $A$ and $B$.
  But now, the ordering $a_i \le' a_j$ is definable on $A$ via the 0-definable $L^*$-formula
$(\forall b\in B)[\psi(a_j,b)\rightarrow\psi(a_i,b)]$.  Then $(A,\le')$ witnesses that there is a type-respecting, definable embedding of $(\Q,\le)$  into $M^*$.  

(3)   By Lemma \ref{nonwm}, $T$ is not weakly minimal, so the following will suffice.

\begin{fact} \label{wm}  If $T$ is stable but not weakly minimal, then, working in a large, saturated model $\CC$ of $T$,
 there is a model $M\preceq \CC$ and singletons $a,b$ such that $\tp(a/Mb)$ is not algebraic, but forks over $M$.
\end{fact}

\begin{proof}  As $T$ is not weakly minimal, there are $M_0\preceq N$ and $p\in S_1(M_0)$ that has two non-algebraic extensions to $S_1(N)$.  As $p$ is stationary,
this implies there is a non-algebraic $q\in S_1(N)$ that forks over $M_0$.  Let $a$ be any realization of $q$, and choose $Y$ to be maximal such that $M_0\subseteq Y\subseteq N$
and $\fg a {M_0} Y$.  As $\tp(a/N)$ forks over $M_0$, $Y\neq N$, so choose any singleton $b\in N\setminus Y$.  By the maximality of $Y$, $\nfg a Y b$.
To complete the proof, choose a model $M\supseteq Y$ with $\fg M Y {ab}$.  It follows by symmetry and transitivity of non-forking that $\nfg a M b$.
Also, since $\tp(a/N)$ is non-algebraic, so is $\tp(a/Yb)$.  But, as $\tp(a/M)$ does not fork over $Yb$, $\tp(a/M)$ is non-algebraic as well.
\end{proof}

Fix $a,b,M$ as in Fact~\ref{wm} and choose an formula $\phi(x,y)\in \tp(ab/M)$ (with parameters from $M$) that witnesses the forking over $M$.

Let $r=\tp(b/M)$ and choose a Morley sequence  $B=\{b_n:n\in\omega\}$  in $r$.  Let $q=\stp(a/Mb)$ and, for each $n$, let $q_{b_n}$ be the strong type over $Mb_n$ conjugate to $q$.
Recursively construct sets $\{I_n:n\in\omega\}$ where each $I_n=\{a_{n,m}:m\in\omega\}$ is a Morley sequence of realizations of the non-forking extension $q^*_{b_n}$ of $q_{b_n}$
to $M\cup B\cup \bigcup\{I_k:k<n\}$. 
It follows by symmetry and transitivity of non-forking that each $I_n$ is independent and fully indiscernible over $MB\cup\bigcup\{I_k:k\neq n\}$. 

Let $A=\{a_{n,m}:n,m\in\omega\}$.  
Now,
any permutation $\sigma\in Sym(B)$ is $L_M$-elementary, and in fact, induces an $L_M$-elementary permutation $\sigma^*\in Sym(AB)$.
Let $L^*=L\cup\{A,B,C_1,\dots,C_n\}$ and let $\CC^*$ be the natural monadic expansion of $\CC$ formed by interpreting
$A$ and $B$ as above, and interpreting each  $C_i$ as $\{c_i\}$, where $\{c_1,\dots,c_n\}$ are the parameters occurring in $\phi$. [We silently replace $\phi(x,y)$
by the natural 0-definable $L^*$-formula formed by replacing each $c_i$ by $C_i$.]  
Finally, define an $L^*$-definable binary relation $E$ on $A^2$ by:
$$E(a,a')\quad\Longleftrightarrow\quad (\exists b\in B)[\phi(a,b)\wedge\phi(a',b)]$$
It is easily checked that $E$ is an equivalence relation, whose classes are precisely $\{I_n:n\in\omega\}$.
Thus, $(A,E)$ is the image of a type-respecting, definable embedding of $\E$ into $\CC^*$.

(4)  We prove this by cases.  If $T$ is not monadically NIP, then $\RR$ definably embeds in a unary expansion, and expanding by a further unary predicate naming infinitely many infinite cliques with no edges between them definably embeds $\E$, so assume $T$ is monadically NIP.  If $T$ is also monadically stable, we are done by (3), so assume $T$ is not monadically stable.  Then, by (2),
 there is a type-respecting, definable embedding of $(\Q,\le)$ into some monadic expansion $M^*$ of a model of $T$.  Thus, it suffices to prove that $\E$ definably embeds into some monadic expansion of 
$(\Q,\le)$.  But this is easy. 
Let $A = \Q \bs \Z$. Then $A$ is 0-definable in the monadic expansion $(\Q,\le, A)$, as is the relation $E\subseteq A^2$ given by
$$E(a,a')\quad\Longleftrightarrow\quad \forall x([a<x<a' \vee a'<x<a]\rightarrow A(x))$$
It is easily checked that $(A,E)$ is isomorphic to $\E$.
\end{proof}

We close this section by stating one `improvement' of Theorem~\ref{paradigm}(4) that will be used in Section~\ref{expand}.
Whereas Theorem~\ref{paradigm} speaks about a definable embedding of $\E$ into some monadic expansion of some model of $T$, we isolate
the following corollary, which describes a weaker configuration that can be found in arbitrary models of $T$ in the original language.

\begin{corollary}  \label{config} Suppose $T$ is a complete $L$-theory that is monadically NIP, but not monadically NFCP.
Then there is an $L$-formula $\phi(x,y,\zbar)$ such that, for every model $N$ of $T$ and every $n\ge 1$, there is $\dbar_n$ and disjoint sets $B_n=\{b_i^n: i < n\}$, 
$A_n=\{a^n_{i,j}:i,j<n\}$ that are without repetition such that
\begin{enumerate}
\item  The sets $\{A_n,B_m:n,m\in\omega\}$ are pairwise disjoint;
\item  For all $n$, all $i,j,k<n$, one of the following holds.
\begin{enumerate}
\item $T$ is stable and 
$N\models\phi(b^n_k,a^n_{i,j})$ if and only if $k=i$; 
\item $T$ is unstable and $N\models\phi(b^n_k,a^n_{i,j})$ if and only if $k\le i$.
\end{enumerate}
\end{enumerate}
Moreover, we may additionally assume that the set $X=N\setminus \bigcup_{n\ge 1}(A_n\cup B_n)$ is infinite.
\end{corollary}

\begin{proof}  As in the proof of Theorem~\ref{paradigm}(2),(3), we split into cases depending on whether or not $T$ is stable.
If $T$ is unstable, as in the proof of Theorem~\ref{paradigm}(2),  choose an $L$-formula $\phi(x,y,\zbar)$ witnessing the order property in large, sufficiently saturated models of $T$.
Now, choose any $N\models T$.  As there is some sufficiently saturated $N'\succeq N$ in which $\phi(x,y,\dbar)$ codes the order property, it follows from elementarity that,
for any fixed $n$, there are $\dbar_n\in N^{\lg(\zbar)}$ and disjoint sets $\{b_i:i<n\}$ and $\{a_{i,j}:i,j<n\}$ such that for all $k,i,j<n$, $N\models \phi(b_k,a_{i,j},\dbar_n)$
if and only if $k\le i$.  

To get the pairwise disjointness, note that if $\{b_i:i<n\}$, $\{a_{i,j}:i,j<n\}$ work for $n$, then for any subset $s\subseteq n$, the subsets $\{b_i:i\in s\}$, $\{a_{i,j}:i,j\in s\}$ work for
$n'=|s|$.
Thus, given any fixed finite set $F$ to avoid, given any $n$, by choosing $m\ge n$ large enough and choosing an appropriate $s\subseteq m$, we can find
disjoint sets $\{b_i:i<n\}$ and $\{a_{i,j}:i,j<n\}$, each of which are disjoint from $F$.

Using this, we can recursively define sequences $\dbar_n$ and
pairwise disjoint families $B_n=\{b_i^n:i<n\}$ and $A_n=\{a^n_{i,j}:i,j<n\}$ such that for all $k,i,j<n$, $N\models \phi(b_k,a_{i,j},\dbar_n)$
if and only if $k\le i$.  By passing to an infinite subsequence, using the remarks above,  and reindexing
we can shrink any family $\{B_n,A_n:n\in\omega\}$ to one satisfying the Moreover clause.

If $T$ is stable, then $T$ is not weakly minimal by Lemma \ref{nonwm}.
Thus, as in the proof of  Theorem~\ref{paradigm}(3),  there is a sufficiently saturated elementary extension $N'\succeq N$ and a formula $\phi(x,y,\zbar)$ that witnesses forking, such that in $N'$ there are $\{b_i:i\in \omega\}$, $\{a_{i,j}:i,j\in\omega\}$, and $\dbar$ such that for all $i,j,k\in \Z$, $N'\models\phi(b_k,a_{i,j},\dbar)$ if and only if $k=i$.

Now, using this configuration, the methods used in the unstable case apply here as well.
\end{proof}

\section{Sets definable in purely monadic and monadically NFCP structures}

 Fact~\ref{MAchar} asserts that a theory is monadically NFCP if and only if it is mutually algebraic, so we recall what is known about sets definable in a mutually algebraic structure.
Throughout this section, fix an infinite cardinal $\lambda$ and think of the set $\lambda=\{\alpha:\alpha\in\lambda\}$ as being the universe of a structure.

\begin{definition}  \label{MA} Fix any infinite cardinal $\lambda$ and any integer $k\ge 1$. 
\begin{itemize}
\item A subset $Y\subseteq \lambda^k$ is {\em mutually algebraic} if there is some integer $m$ so that for every $a\in \lambda$,
$\{\abar\in Y:a\in\abar\}$ has size at most $m$.  
\item  A subset $Y^*\subseteq \lambda^{k+\ell}$ is {\em padded mutually algebraic} if, for some permutation $\sigma\in Sym(k+\ell)$ of the coordinates,
 there is a mutually algebraic $Y\subseteq \lambda^k$ and $Y^*=\sigma(Y\times \lambda^\ell)$.
 \item  A model $M$ with universe $\lambda$ is {\em mutually algebraic} if, for every $n$, every definable (with parameters) $D\subseteq \lambda^n$ is a boolean combination of definable (with parameters) padded mutually algebraic sets.
 \item  A complete theory $T$ is {\em mutually algebraic} if some (equivalently, all) models of $T$ are mutually algebraic.
 \end{itemize}
 \end{definition}
 
 Trivially, every unary subset $Y\subseteq \lambda^1$ is mutually algebraic.  
 
 \begin{fact}[\cite{LTArray}*{Theorem 2.1}] \label{MAcrit}  An $L$-structure $M$ is mutually algebraic if and only if every
 atomic $L$-formula $\alpha(x_1,\dots,x_n)$ is equivalent to a boolean combination of quantifier-free definable (with parameters) padded mutually algebraic sets.
 \end{fact}

 It follows immediately that any purely monadic structure is mutually algebraic.

In this section, our goal is to obtain a particular configuration, described in Lemma \ref{monthree}, appearing in any mutually algebraic structure whose theory is not purely monadic. This will be used in the proof of Theorem \ref{big}, when a non-monadically definable $Y$ induces a mutually algebraic structure.
 
We begin by characterizing which mutually algebraic sets $Y\subseteq \lambda^k$ are monadically definable.  
Obviously, every $Y\subseteq\lambda^1$ is monadically definable, so we concentrate on $k\ge 2$.
As notation, let $\Delta_k=\{(a,a,\dots,a)\in\lambda^k:a\in\lambda\}$ denote the set of constant $k$-tuples.

\begin{lemma}  \label{monone}   Fix any infinite cardinal $\lambda$ and any integer $k\ge 2$.  A mutually algebraic subset $Y\subseteq \lambda^k$ is 
monadically definable if and only if $Y\setminus \Delta_k$ is finite.
\end{lemma}

\begin{proof}  First, suppose $Y\setminus\Delta_k$ is finite.  Let $F=\bigcup(Y\setminus\Delta_k)=\{a_1,\dots,a_n\}\subseteq\lambda$ and let $Z=\{a\in\lambda:(a,a,\dots,a)\in Y\}$.
Let $N=(\lambda,U_1,\dots,U_n,U_{n+1})$ be the structure in which $U_i$ is interpreted as $\{a_i\}$ for each $i\le n$ and $U_{n+1}$ is interpreted as $Z$.
Then $Y$ is definable in $N$, so $Y$ is monadically definable.  

Conversely, suppose $Y$ is mutually algebraic and definable in some monadic $N=(\lambda,U_1,\dots,U_n)$.  It is easily seen that $N$ admits elimination of quantifiers.
Collectively, the unary predicates $U_i$ color each element $a\in\lambda$ into one of $2^n$ colors.  Some of these $2^n$ colors will have infinitely many elements of $\lambda$, while other colorings have only finitely many elements.  Let $F=\{a\in\lambda:$ there are only finitely many $b\in\lambda$ such that $N\models\bigwedge_{i=1}^n U_i(a)\leftrightarrow U_i(b)\}$.
The set $F$ is clearly finite.  Now, the elements of $\lambda\setminus F$ are partitioned into finitely many infinite chunks, each of which is fully indiscernible over its complement.
Thus, it follows that $F=\acl_N(\emptyset)$ and for any $a\in \lambda$, $\acl_N(a)=F\cup\{a\}$.  To  show $Y\setminus \Delta_k$ finite, it suffices to prove the following.

\begin{claim*}  $Y\subseteq F^k\cup \Delta_k$.
\end{claim*}

\begin{claimproofend}  Choose any $\abar\in\lambda^k\setminus (F^k\cup\Delta_k)$.  Since $\abar\not\in F^k$, choose a coordinate $a^*\in\abar$ with $a^*\not\in F$.
Since the $k$-tuple $\abar$ is not constant, choose $b\in\abar$ with $b\neq a^*$.  Now, by way of contradiction, suppose $\abar\in Y$.  As $Y$ is mutually algebraic,
$a^*\in\acl_N(b)=F\cup\{b\}$, which it isn't.
\end{claimproofend}
\let\qed\relax
\end{proof}

\begin{lemma} \label{montwo}   Suppose $M$ is a mutually algebraic structure with universe $\lambda$ such that $Th(M)$ is not purely monadic.
Then, for some $k\ge 2$ there is some  $L_M$-definable, mutually algebraic $Y\subseteq\lambda^k$ with $Y\setminus \Delta_k$
 infinite.
\end{lemma}

\begin{proof}  
Fix such an $M$ and assume that no such $L_M$-definable, mutually algebraic set existed.  By Lemma~\ref{monone} we would have that for every $k$,
every $L_M$-definable,
mutually algebraic subset of $\lambda^k$ is monadically definable.  From this, it follows easily that every $L_M$-definable, padded mutually algebraic set would
be monadically definable, as would every boolean combination of these. As $M$ is mutually algebraic, it follows that every $L_M$-definable set is monadically definable, contradicting $Th(M)$ not being purely monadic.
\end{proof}

We now obtain our desired configuration.

\begin{lemma} \label{monthree}  Suppose $M$ is a mutually algebraic structure with universe $\lambda$ whose theory is not purely monadic.
Then there is some $k\ge 2$, some $L_M$-definable $Y\subseteq\lambda^k$ and an infinite set $\F=\{\abar_n:n\in\omega\}\subseteq Y\setminus\Delta_k$ such that
\begin{enumerate}
\item  For each $n\in\omega$, $(\abar_n)_1\neq(\abar_n)_2$  (the first two coordinates differ); and
\item  $\abar_n\cap\abar_m=\emptyset$ for distinct $n,m\in\omega$.
\end{enumerate}
In particular,  if $F=\bigcup\F$, then for every $a\in F$ there is exactly one $\abar\in Y$ with $\abar\subseteq F$ (and hence $(\abar)_1\neq(\abar_2)$).
\end{lemma}

\begin{proof}  By Lemma~\ref{montwo}, choose $k\ge 2$ and an $L_M$-definable,  mutually algebraic $Y\subseteq\lambda^k$ such that $X:=Y\setminus\Delta_k$ is infinite.
By mutual algebraicity, choose an integer $K$ such that for every $a\in\lambda$, there are at most $K$ $k$-tuples $\abar\in Y$ with $a\in\abar$.
As each element of $X$ is a non-constant $k$-tuple, by the pigeonhole principle we can find an infinite $X'\subseteq X$ and $i\neq j\in[k]$ such that $(\abar)_i\neq(\abar)_j$ for each $\abar\in X'$.  By applying a  permutation $\sigma\in Sym([k])$ to $Y$, we may assume $i=1$ and $j=2$, so after this transformation (1) holds for any $\abar\in X'$.
But now, as $X'\subseteq Y$ is infinite, while every element $a\in\lambda$ occurs in only finitely many $\abar\in X'$, it is easy to recursively construct $\F=\{\abar_n:n\in\omega\}\subseteq X'$.
\end{proof}

\section{Monadically stable and monadically NIP are aptly named}  \label{expand}

In this section, we prove Theorem~\ref{big}.  
The positive part,  that $(T,Y)$ is always monadically NFCP whenever both $T$ is and $Y\subseteq\lambda^k$ is monadically NFCP definable,
is immediate from the following.

\begin{lemma} \label{positive} Suppose $N_1$ and $N_2$ are structures, both with universe  $\lambda$, in disjoint languages $L_1$ and $L_2$.  If both $N_1$ and $N_2$ are monadically NFCP (=mutually algebraic) then the expansion $N^*=(N_1,N_2)$ is monadically NFCP as well.
\end{lemma}

\begin{proof}  By replacing each function and constant symbol by its graph, we may assume both $L_1$ and $L_2$ only have relation symbols.
As the languages are disjoint, this implies that every $L_1\cup L_2$-atomic formula is either $L_1$-atomic or $L_2$-atomic.  
Thus,  every atomic formula in $N^*$ is either equivalent to a boolean combination of either $L_1$-definable or $L_2$-definable padded, mutually algebraic formulas.
As the notion of a set $Y\subseteq\lambda^k$ being  padded mutually algebraic is independent of any structure, the result follows by applying Fact~\ref{MAcrit}.
\end{proof}

The negative directions are more involved.  To efficiently handle the various cases, we first prove two propositions, from which all of the negative results follow in
Theorem~\ref{negative}.

For the following proposition, first note that a structure with two cross-cutting equivalence relations admits coding. We will essentially encode this configuration, but since we don't want to assume that either $N_1$ or $N_2$ is saturated for our eventual application, we must work with the finitary approximations to an equivalence relation with infinitely many infinite classes provided by Corollary \ref{config}.

\begin{proposition}  \label{prop1}  Suppose $L_1$ and $L_2$ are disjoint languages, $\lambda\ge ||L_1\cup L_2||$ a cardinal,
$N_1$ is an $L_1$-structure with universe $\lambda$, and $N_2$ is an $L_2$-structure with universe $\lambda$.
If both $Th(N_1)$ and $Th(N_2)$ are not monadically NFCP, then there is a permutation $\sigma\in Sym(\lambda)$ such that the
$L_1\cup L_2$-structure $(N_1,\sigma(N_2))$ has a theory that is not monadically NIP.
\end{proposition}

\begin{proof}  We may assume $Th(N_1)$ and $Th(N_2)$ are monadically NIP, since we are finished otherwise. Apply Corollary~\ref{config} to both $N_1$ and $N_2$.  This gives an $L_1$-formula $\phi(x,y,\zbar)$ and, for each $n$, pairwise disjoint sets $A_n=\{\alpha^n_{i,j}:i,j<n\}$,
$B_n=\{\beta^n_{i}:i<n\}$ and $\rbar_n$ as there, with exceptional set $X=\lambda\setminus\bigcup_{n\ge 1} (A_n\cup B_n)$.  Note that as each $A_n,B_n$ is finite, $|X|=\lambda$.
On the $L_2$-side, choose an $L_2$-formula $\psi(x,y,\wbar)$ such that, for all $n\ge 1$, there is $\sbar_n\in \lambda^{\lg(\wbar)}$ and pairwise disjoint sets
$C_n=\{\gamma^n_{i,j}:i,j<n\}$ and $D_n=\{\delta^n_i:i<n\}$ as there.

Now choose $\sigma\in Sym(\lambda)$ to be any permutation satisfying:  For all $n\ge 1$,
\begin{enumerate}
\item  $\sigma(D_n)\subseteq X$; and
\item  $\sigma$ maps $C_n$ bijectively onto $A_n$ via  $\sigma(\gamma^n_{i,j})=\alpha_{j,i}^n$.
\end{enumerate}

Note that there are many permutations $\sigma$ satisfying these constraints.  Choose one, and let
$\sigma(N_2)$ be the unique $L_2$ structure with universe $\lambda$ so that $\sigma$ is an $L_2$-isomorphism.

\begin{claim*}
	 The $L_1\cup L_2$-theory $Th(N_1,\sigma(N_2))$ is not monadically NIP.
\end{claim*}

\begin{claimproofend}  We will produce $M^*$, a monadic expansion of an $L_1\cup L_2$-elementary extension $\Mbar\succeq (N_1,\sigma(N_2))$ that admits coding, which suffices.
To do this, we first note that by compactness, there is an $L_1\cup L_2$-elementary extension $\Mbar\succeq (N_1,\sigma(N_2))$ that contains disjoint sets
$A=\{a_{i,j}:i,j\in \Z\}$, $B=\{b_i:i\in \Z\}$, $D=\{d_j:j\in \Z\}$, and tuples $\rbar,\sbar$ such that, for all $k,i,j\in\Z$,
either (if $Th(N_1)$ is unstable) $\Mbar\models \phi(b_k,a_{i,j},\rbar)$ if and only if $k\le i$,  or (if $Th(N_1)$ is stable) $\Mbar\models \phi(b_k,a_{i,j},\rbar)$ if and only if $k=i$;
and dually,
either (if $Th(N_2)$ is unstable) $\Mbar\models \psi(d_k,a_{i,j},\sbar)$ if and only if $k\le j$, or (if $Th(N_2)$ is stable) $\Mbar\models \psi(d_k,a_{i,j},\sbar)$ if and only if $k=j$.


Now, given $\Mbar$, let $L^*=L_1\cup L_2\cup\{A,B,D\}$ and let $M^*$ be the natural monadic expansion of $\Mbar$ described by $A,B,D$ above.  To show that $M^*$ admits coding, we need to rectify the ambiguity between the stable and unstable cases.
Specifically, we claim that there is an $L^*$-formula $\phi^*(x,y,\zbar)$ such that for all $b_i\in B$, the solution set $\phi^*(b_i,M^*, \rbar)$ is $\{a_{i,j}:j\in \Z\}$.
If $Th(N_1)$ is stable, this is easy: just take $\phi^*(x,y,\zbar):=A(y)\wedge\phi(x,y,\zbar)$.  
However, when $Th(N_1)$ is unstable, we need some more $L^*$-definability in $M^*$.  Specifically, note that in this case, the natural ordering on $B$ is $L^*$-definable via
$$b_i\le b_j \ \hbox{if and only if} \ \forall y [(A(y)\wedge\phi(b_j,y,\rbar))\rightarrow\phi(b_i,y,\rbar)]$$
As the ordering on $B$ is discrete, every element $b\in B$ has a unique successor, $S(b)$, and this operation is $L^*$-definable since $\le$ is.
Using this, the $L^*$-formula
$$\phi^*(x,y,\zbar):=B(x)\wedge A(y)\wedge\phi(x,y,\zbar)\wedge\neg\phi(S(x),y,\zbar)$$
is as desired. 

Arguing similarly, there is an $L^*$-formula $\psi^*(x,y,\wbar)$ such that for all $d_j\in D$, the solution set $\psi^*(d_j,M^*,\sbar)$ is $\{a_{i,j}\in A:i\in\Z\}$.
Putting these together, let $\theta(u,v,y,\zbar,\wbar)$ be the $L^*$-formula
$$B(u)\wedge D(v)\wedge A(y)\wedge \phi^*(u,y,\zbar)\wedge \psi^*(v,y,\wbar)$$
Then the solution set of $\theta(u,v,y,\rbar,\sbar)$ is precisely the graph of a bijection from $B\times D$ onto $A$.
Thus, $M^*$ admits coding, which suffices.
\end{claimproofend}
\let\qed\relax
\end{proof}

The proof of the next proposition is in many ways similar. Here our ideal infinitary configuration consists of an equivalence relation with infinitely many infinite classes, with each tuple from the configuration in Lemma \ref{monthree} pairing two classes by intersecting them. But again, instead of our ideal equivalence relation, we must restrict ourselves to the finitary approximations from Corollary \ref{config}.

\begin{proposition}  \label{prop2}  Suppose $L_1$ and $L_2$ are disjoint languages, $\lambda\ge ||L_1\cup L_2||$ a cardinal,
$N_1$ is an $L_1$-structure with universe $\lambda$, and $N_2$ is an $L_2$-structure with universe $\lambda$.
If $Th(N_1)$ is not monadically NFCP, and if $Th(N_2)$ is monadically NFCP but not purely monadic,
then there is a permutation $\sigma\in Sym(\lambda)$ such that the
$L_1\cup L_2$-structure $(N_1,\sigma(N_2))$ has a theory that is not monadically NIP.
\end{proposition}

\begin{proof}  We may assume $Th(N_1)$ is monadically NIP, since we are finished otherwise. Apply Corollary~\ref{config} to $N_1$, obtaining an $L_1$-formula $\phi(x,y,\zbar)$ and, for each $n$, pairwise disjoint sets $A_n=\{\alpha^n_{i,j}:i,j<n\}$,
$B_n=\{\beta^n_{i}:i<n\}$ and $\rbar_n$ as there, with exceptional set $X=\lambda\setminus\bigcup_{n\ge 1} (A_n\cup B_n)$.  Note that as each $A_n,B_n$ is finite, $|X|=\lambda$.
For the $N_2$ side, apply Lemma~\ref{monthree}, getting an $N_2$-definable $Y\subseteq \lambda^k$ and a distinguished set $\F=\{\ebar_\ell:\ell\in\omega\}\subseteq Y$
as there.  Say $Y$ is defined using parameters $\{c_1,\dots,c_n\}$.
Let $L_2^V=L_2\cup \{V,C_1,\dots,C_n\}$ and let $N_2^V$ be the monadic expansion of $N_2$, interpreting $V$ as $F=\bigcup\F$ and each
$C_i$ as $\{c_i\}$.  Note that in $N_2^V$, the subsets
$F_1=\{(\ebar)_1:\ebar\in\F\}$, $F_2=\{(\ebar)_2:\ebar\in\F\}$ of $F$ are $L_2^V$-definable (without parameters), along with the bijection $f:F_1\rightarrow F_2$ given by:  $f(x)=(\ebar)_2$, where $\ebar$ is the unique element of $\F$ containing $x$.  Fix an enumeration $\{\gamma_\ell:\ell\in\omega\}$ of $F_1\subseteq\lambda$.

We now choose a permutation $\sigma\in Sym(\lambda)$ that satisfies:  
\begin{itemize}
\item  For all $n\ge 1$ and all distinct $i<j<n$, there is some (in fact, unique) $\ell\in\omega$ such that $\sigma(\gamma_\ell)=\alpha^n_{i,j}$ and $\sigma(f(\gamma_\ell))=\alpha^n_{j,i}$.
\end{itemize}

Let $\sigma(N_2^V)$ be the $L_2^V$-structure with universe $\lambda$ so that $\sigma$ is an $L_2^V$-isomorphism and let $M_0^V=(N_1,\sigma(N_2^V))$ be the expansion of 
$N_1$ to an $L_1\cup L_2^V$-structure.  So $M_0^V$ has universe $\lambda$ and satisfies:  
\begin{itemize}
\item For all $n\ge 1$ and $i<j<n$, $f(\alpha^n_{i,j})=\alpha^n_{j,i}$; and
\item  The relationships given by $N_1$.
\end{itemize}
Let $M_0$ be the $L_1\cup L_2$-reduct of $M_0^V$.

\begin{claim*}
	The $L_1\cup L_2$-theory of $M_0$ is not monadically NIP.
\end{claim*}  

\begin{claimproofend}  We show that the $L_1\cup L_2^V$-theory of $M_0^V$ is not monadically NIP, which suffices.  For this, 
the strategy is similar to the proof of Proposition~\ref{prop1}.   We will find an $L_1\cup L_2^V$-elementary extension $\Mbar$ of $M_0^V$ and then
find a monadic expansion $M^*$ of $\Mbar$ that admits coding.  
Specifically, choose an $L_1\cup L_2\cup \{V\}$-elementary extension $\Mbar$ for which there are sets
$B=\{b_i:i\in\Z\}$, $A=\{a_{i,j}:i \neq j\in\Z\}$ such that
\begin{enumerate}
	\item  For all $i<j$ from $\Z$, $f(a_{i,j})=a_{j,i}$.
\item One of the following holds.
\begin{enumerate}
	\item $Th(N_1)$ is unstable, and $\Mbar\models \phi(b_k,a_{i,j},\rbar)$ if and only if $k\le i$.
	\item $Th(N_1)$ is stable, and $\Mbar\models \phi(b_k,a_{i,j},\rbar)$ if and only if $k=i$.
\end{enumerate}
\end{enumerate}

Given such an $\Mbar$, let $L^*=L_1\cup L_2^V\cup\{A,B\}$, and let $M^*$ be the expansion of $\Mbar$ interpreting $A$ and $B$ as themselves.
Exactly as in the proof of Proposition~\ref{prop1}, find an $L^*$-formula $\phi^*(x,y,\zbar)$ such that for all $b_i\in B$, the solution set $\phi^*(b_i,M^*,\rbar)$ is $\{a_{i,j}:j\in \Z, j \neq i\}$.
Finally, let $L^+ = L^* \cup \set{B^-, B^+, A^*}$ with $B^- = \{b_i:i\in\Z^{<0}\}$, $B^+ = \{b_i:i\in\Z^{>0}\}$, and $A^* = \{a_{i,j} : i \in Z^{<0}, j \in \Z^{>0}\}$. Let $\theta(u,v,y, \zbar)$ be the $L^+$-formula
$$B^-(u)\wedge B^+(v)\wedge A^*(y)\wedge \phi^*(u,y,\zbar)\wedge \phi^*(v,f(y),\zbar)$$
Then the formula $\theta(u,v,y,\rbar)$ is the graph of a bijection from $B^- \times B^+ \to A^*$, which suffices.
\end{claimproofend}
\let\qed\relax
\end{proof}

Using Propositions~\ref{prop1} and \ref{prop2} we are now able to prove the negative portions of Theorem~\ref{big}.
As the positive portion was proved in Lemma~\ref{positive}, this suffices.

\begin{theorem}  \label{negative}  Suppose $T$ is a complete $L$-theory and $Y\subseteq \lambda^k$ with $\lambda\ge ||L||$.  Then:
\begin{enumerate}
\item  If $T$ is not monadically NFCP and $Y$ is not monadically definable, then $(T,Y)$ is not always monadically NIP; and
\item  If $T$ is not purely monadic and $Y$ is not monadically NFCP definable, then $(T,Y)$ is not always monadically NIP.
\end{enumerate}
\end{theorem}

\begin{proof}  (1)  Choose $N_1\models T$ with universe $\lambda$, and let $N_2=(\lambda,Y)$ be the structure in the language $L_2=\{Y\}$ with the obvious interpretation.
Now, depending on whether or not $Th(N_2)$ is monadically NFCP or not, apply either Proposition~\ref{prop1} or Proposition~\ref{prop2} to get a permutation $\sigma\in Sym(\lambda)$ such that $Th(N_1,\sigma(N_2))$ is not monadically NIP.  
Of course, $Y$ need not be preserved here, so apply $\sigma^{-1}$.  That is, let $(\sigma^{-1}(N_1),Y)$ be the $L\cup\{Y\}$-structure so that $\sigma^{-1}$ is an $L\cup\{Y\}$ isomorphism.  As $\sigma(N_1)\models T$, this structure witnesses that $(T,Y)$ is not always monadically NIP.

(2)  Let $N_1=(\lambda,Y)$ and let $N_2$ be any model of $T$ with universe $\lambda$.  
Again, by either Proposition~\ref{prop1} or Proposition~\ref{prop2} (depending on $Th(N_2)$), we get a permutation $\sigma\in Sym(\lambda)$
such that $(N_1,\sigma(N_2))$ has a non-monadically NIP theory.  But this structure is precisely $(\sigma(N_2),Y)$ and $\sigma(N_2)\models T$, so again $(T,Y)$ is not always
monadically NIP.
\end{proof}

        \bibliographystyle{alpha}
\bibliography{../../../Bib.bib}

\end{document}